\documentclass[oneside]{amsart}
\usepackage[left=2cm,top=2cm,right=3cm,bottom=1cm, footskip=.5cm]{geometry}
\usepackage{amsmath}
\usepackage{graphicx}
\usepackage{tikz}
\usepackage{fancyvrb}
\usepackage{pdfpages}
\usepackage[makeroom]{cancel}
\usepackage{amssymb} 
\usepackage[parfill]{parskip}
\usepackage{mathtools}
\usepackage{hyperref}
\usepackage{tikz-cd}
\usepackage{setspace}
\usepackage{amsthm}
\usepackage{cite}

\newtheoremstyle{colon}%
{}
{}
{\itshape}
{}
{\bfseries}
{:}
{ }
{}

\theoremstyle{colon}

\newtheorem{lemma}{Lemma}[section]
\newtheorem{theorem}[lemma]{Theorem}

\newtheorem{corollary}[lemma]{Corollary}
\newtheorem{definition}[lemma]{Definition}
\newtheorem{example}[lemma]{Example}

\title{A criterion for a Hurewicz cofibration to be a Quillen cofibration}
\author{Andrew Ronan}
\date{}

\begin{document}

	\maketitle
	
	\begin{abstract}
		In this paper, we prove that $h$-cofibrations between $q$-cofibrant spaces are $q$-cofibrations. We also present a number of applications, including a pushout-product property for symmetrizable cofibrations, a local-to-global gluing lemma for $q$-cofibrations, a proof that $q$-fibrations between $q$-cofibrant spaces are $h$-fibrations, and an alternative proof of Waner's theorem on $G$-spaces with the $G$-homotopy type of a $G$-CW complex in fibre sequences. Moreover, all of the above generalises readily to the equivariant context, and so we work in the more general equivariant setting throughout.
	\end{abstract}

\tableofcontents

\section{Introduction}
	
	In this paper, we prove that $hG$-cofibrations between $qG$-cofibrant $G$-spaces are $qG$-cofibrations and discuss a number of applications. Paired with Illman's theorem that a $G$-manifold admits the structure of a $G$-CW complex, \cite[Corollary 7.2]{I83}, our result proves to be a powerful recognition theorem for $qG$-cofibrations, as we will see in our applications. 
	
	At first glance, our result might seem surprising - after all, $hG$-cofibrations can often be determined by local information. For example, the inclusion of any subspace, $A$, admitting a neighbourhood homeomorphic to $A \times(-1,1)$, with $A$ corresponding to $A \times \{0\}$, is an $hG$-cofibration. On the other hand, the data of a $qG$-cofibration seems to require global information in the sense of a $G$-cell complex connecting the domain and codomain. Of course, $qG$-cofibrations are only defined as retracts of $G$-cell complexes which is perhaps the first sign of this intuition breaking down. For another perspective, note that we are in a situation where we have two model structures, the $hG$ and $qG$ model structures, satisfying the relations $\mathcal{W}_h \subset \mathcal{W}_q$ and $\mathcal{F}_h \subset \mathcal{F}_q$ between weak equivalences and fibrations. In such a situation we can form a mixed model structure as demonstrated by Cole in \cite[Theorem 2.1]{C06}, which in this case we call the $mG$-model structure. It is a general fact about mixed model structures that $h$-cofibrations between $m$-cofibrant objects are $m$-cofibrations, \cite[Proposition 3.11]{C06}. However, it is not true in general that an $h$-cofibration between $q$-cofibrant objects is a $q$-cofibration, as the following example shows:

	\begin{example}
		By \cite[Theorem 4.3]{AB25}, we have two model structures on the category of sets given by the triples of weak equivalences, cofibrations and fibrations, $(\mathcal{W}_h, \mathcal{C}_h, \mathcal{F}_h)$ and $(\mathcal{W}_q, \mathcal{C}_q, \mathcal{F}_q)$, defined as follows. We let both $\mathcal{W}_h$ and $\mathcal{W}_q$ be the class of all maps between sets. We define $C_h$ to be the class of bijections and $C_q$ to be the class of injections. We define $F_h$ to be the class of all maps between sets and $F_q$ to be the class of all surjections. We therefore have that $\mathcal{W}_h \subset \mathcal{W}_q, \mathcal{C}_h \subset \mathcal{C}_q$ and $\mathcal{F}_q \subset \mathcal{F}_h$. Then an $h$-fibration between $q$-fibrant objects is simply a map between non-empty sets, which need not be a $q$-fibration (that is, surjective). These model structures define dual model structures on the category $\textbf{Set}^{op}$ given by $(\mathcal{W}_h^{op}, \mathcal{C}_h^{op}, \mathcal{F}_h^{op})$ and $(\mathcal{W}_q^{op}, \mathcal{C}_q^{op}, \mathcal{F}_q^{op})$ where $\mathcal{W}_h^{op} = \mathcal{W}_h$, $\mathcal{C}_h^{op} = \mathcal{F}_h$, $\mathcal{F}_h^{op} = \mathcal{C}_h$, $\mathcal{W}_q^{op} = \mathcal{W}_q$, $\mathcal{C}_q^{op} = \mathcal{F}_q$ and $\mathcal{F}_q^{op} = \mathcal{C}_q$. Therefore, $\mathcal{W}_h^{op} \subset \mathcal{W}_q^{op}$ and $\mathcal{F}_h^{op} \subset \mathcal{F}_q^{op}$. However, an $h$-cofibration between $q$-cofibrant sets in $\textbf{Set}^{op}$ is dual to an $h$-fibration between $q$-fibrant sets in $\textbf{Set}$, which is not necessarily a $q$-fibration. So, an $h$-cofibration between $q$-cofibrant sets in $\textbf{Set}^{op}$ is not necessarily a $q$-cofibration.
	\end{example}

Our theory works just as well using the $\mathcal{C}$-model structure of \cite[Proposition B.7]{S18}, which is a generalisation of the $qG$-model structure, so we work with the $\mathcal{C}$-model structure throughout. An early application is showing that a $G$-equivariant open subspace of a $\mathcal{C}$-cofibrant space is $\mathcal{C}$-cofibrant, and this leads to the following sharpening of our main result:	
	
	\begin{theorem} \label{hco71}
		If $i: A \to B$ is an $hG$-cofibration and $B$ is $\mathcal{C}$-cofibrant, then $A$ is $\mathcal{C}$-cofibrant and $i$ is a $\mathcal{C}$-cofibration.	
	\end{theorem}

In fact, some of our applications do need the full generality of the $\mathcal{C}$-model structure. For example, in Section \ref{hco81} we let $\mathcal{C}$ be the set of isotropy groups of $[0,1]^n$, viewed as a $\Sigma_n$-space, to obtain the following pushout-product property of symmetrizable cofibrations:
	
	\begin{theorem}
		Let $j: A \to B$ be a symmetrizable cofibration in a tensored and cotensored closed symmetric monoidal topological model category and let $i: C \to D$ be a $q$-cofibration of spaces.  Then the pushout product map:	
		\[
		j \widehat{\times} i : A \times D \cup_{A \times C} B \times C \to B \times D
		\]
		
		is a symmetrizable cofibration, which is a symmetrizable acyclic cofibration if either $i$ is acyclic or $j$ is a symmetrizable acyclic cofibration.
	\end{theorem}

In Section \ref{hco82}, we derive local-to-global gluing results for $\mathcal{C}$-fibrations and $\mathcal{C}$-cofibrations. We previously showed, in Lemma \ref{hco10}, that if $i: A \to B$ is a $\mathcal{C}$-cofibration and $U \subset B$ is a $G$-equivariant open subspace, then $i: A \cap U \to U$ is a $\mathcal{C}$-cofibration.  One consequence is the following local-to-global gluing result for $\mathcal{C}$-cofibrations:

\begin{theorem} \label{hco73}
	If $i: A \to B$ is a map of $G$-spaces and there is a $G$-numerable open cover $\mathcal{U}$ of $B$ such that for every $U \in \mathcal{U}$, $i|_{A \cap U}: A \cap U \to U$ is a $\mathcal{C}$-cofibration, then $i$ is a $\mathcal{C}$-cofibration.
\end{theorem}
	
	In Section \ref{hco83}, we kick off a new chain of applications, by considering $n$-fold products and the diagonal map. In particular, we derive the following equivariant analogue of \cite[Corollary III.2]{DE72}:
	
	\begin{theorem} \label{hco74} If $X$ is $qG$-cofibrant, then the diagonal map $\Delta: X \to X^n$ is a $q(G \times \Sigma_n)$-cofibration.
	\end{theorem}

In Section \ref{hco84}, we use our results from Section \ref{hco83} to help prove that a $q$-fibration between $q$-cofibrant spaces is an $h$-fibration, or more generally:

\begin{theorem} \label{hco75}
	Suppose that $\mathcal{C}$ is closed under conjugacy and intersections.	If $p: E \to B$ is a $\mathcal{C}$-fibration between $\mathcal{C}$-cofibrant spaces, then $p$ is an $hG$-fibration.
\end{theorem}

Non-equivariantly, and when both $E$ and $B$ are CW-complexes, Theorem \ref{hco75} is due to Cauty, \cite{C92}, Steinberger and West, \cite[Theorem 1]{SW84}. However, the generalisation to when $E$ and $B$ are $q$-cofibrant, or indeed cell complexes, is a new consequence of our result that open subspaces of $q$-cofibrant spaces are $q$-cofibrant. Equivariantly, and when $E$ and $B$ are both $G$-CW complexes, the result is due to Gevorgyan and Jimenez, \cite[Theorem 6]{GJ19}.

We close in Section \ref{hco85} with one of the most interesting applications of our theory, namely a slick proof of Waner's theorem and its converse:

\begin{theorem} \label{hco76}
	Let $\mathcal{C}$ be closed under conjugacy and intersections. Let $p: E \to B$ be an $hG$-fibration such that $B$ is $m \mathcal{C}$-cofibrant. Then $E$ is $m \mathcal{C}$-cofibrant iff for every $H$ and $b \in B$ with isotropy group $H$, $p^{-1}(b)$ is $m (H \cap \mathcal{C})$-cofibrant. \end{theorem}

Non-equivariantly, Theorem \ref{hco76} is due to Stasheff, \cite{S63}, with an alternative proof given by Sch{\"o}n, \cite{S77}. Equivariantly, the forward direction is due to Waner, \cite[Corollary 4.14]{W80}.

	\textbf{Notations and Conventions}
	
	All spaces are assumed to be CGWH, all groups are assumed to be compact Lie, and all subgroups of compact Lie groups under discussion are assumed to be closed. When working in the category of $G$-spaces, all maps, homotopies, spaces etc. are assumed to be $G$-equivariant, even if we neglect to say so.
	
	Throughout we let $\mathcal{C}$ denote a set of subgroups of $G$ which we take to be closed under conjugacy, and we then work with the $\mathcal{C}$-model structure on the category of $G$-spaces, \cite[Proposition B.7]{S18}. Recall that the $\mathcal{C}$-model structure is a topological model structure cofibrantly generated by:	
\begin{gather*}
	I = \{G/H \times S^{n-1} \to G/H \times D^n | n \geq 0, H \in \mathcal{C}\}, \\
	J = \{G/H \times D^n \to G/H \times D^n \times I | n \geq 0, H \in \mathcal{C}\}
\end{gather*}
	
	Moreover, the $\mathcal{C}$-fibrations (resp. $\mathcal{C}$-equivalences) are the maps $p$ such that $p^H$ is a $q$-fibration (resp. $q$-equivalence) for every $H \in \mathcal{C}$. Note that our assumption that $\mathcal{C}$ is closed under conjugacy results in no loss of generality, but is convenient for proofs.

	\section{Hurewicz cofibrations between $q$-cofibrant spaces are $q$-cofibrations}
	
We now move onto the proof of our main theorem.	If $f,g: X \to Y$ are maps and $i: A \to X$ is a map, we say that a homotopy $H$ between $f$ and $g$ is rel $i$ if $H \circ (i \times 1)$ is the constant homotopy between $fi$ and $gi (= fi)$. Similarly, if $p: Y \to Z$ is a map we say that $H$ is rel $p$ if $pH$ is the constant homotopy between $pf$ and $pg$. We leave it to the reader to determine from context whether a homotopy is relative to precomposition or postcomposition.

 The key point is the following generalisation of \cite[Proposition 17.1.4]{MP12}, which admits the same proof:
	
	\begin{lemma} \label{hco1} Let $i: A \to B$ be an $hG$-cofibration and $C$ a subspace of $A$. Let $p: X \to Y$ be a map which has the RLP wrt $B \times \{0\} \cup_{C \times \{0\}} C \times I \to B \times I$. Suppose that we are given a commutative square:
		
		\[	\begin{tikzcd}
			A \arrow[swap]{d}{i} \arrow{r}{v} & X \arrow{d}{p} \\
			B \arrow[swap]{r}{w} & Y 		
		\end{tikzcd} \]
		
		and a map $\sigma: B \to X$ such that $p \sigma i = pv$, $\sigma i \simeq v$ rel $p$ and $p \sigma \simeq w$ rel $i$. Then $\sigma$ is homotopic to some $\tau$ with $\tau i = v$ and $p \tau = w$.
		
	\end{lemma}
	
	\begin{proof}
		Since $i$ is an $hG$-cofibration, $\sigma$ is homotopic to some $\sigma^{'}$ with $\sigma^{'} i = v$ in such a way that the composite homotopy $p \sigma^{'} \simeq w$ is still relative to $i$. So we have reduced to the case where $\sigma i = v$. Let $(H, \lambda)$ represent $(B,A)$ as an $NDR$-pair. Let $K: B \times I \to Y$ be a homotopy between $p \sigma$ and $w$, rel $i$ and normalised so that $K(b,t) = K(b,1)$ for $t \geq \lambda(b)$ (see Lemma \ref{hco50}). Since $p$ has the RLP wrt $B \times \{0\} \cup_{C \times \{0\}} C \times I \to B \times I$, and using the constant homotopy on $C \times I$, we can lift this homotopy to $\tilde{K}: B \times I \to X$ starting at $\sigma$. If we define $\tilde{\sigma} (b) = \tilde{K}(b,\lambda(b))$ we solve the lifting problem as desired.
	\end{proof}

We can now prove that an $hG$-cofibration between $\mathcal{C}$-cofibrant $G$-spaces is a $\mathcal{C}$-cofibration, or more generally:

\begin{lemma} \label{hco2}
	If we have a commutative triangle:
	
\[	\begin{tikzcd}
		& C \arrow[swap]{dl}{j} \arrow{dr}{k} & \\
		A \arrow[swap]{rr}{i} & & B
	\end{tikzcd} \]

such that $i$ is an $hG$-cofibration and both $j$ and $k$ are $\mathcal{C}$-cofibrations, then $i$ is a $\mathcal{C}$-cofibration.
\end{lemma}

\begin{proof}
In any model category, to prove that a map $i$ between cofibrant objects is a cofibration, it suffices to show that $i$ has the LLP wrt all acyclic fibrations between cofibrant objects, by the retract argument. So consider the $\mathcal{C}$-model structure on $(C \downarrow \textbf{G-Sp})$ induced by the $\mathcal{C}$-model structure on $\textbf{G-Sp}$, \cite[Theorem 7.6.5]{H03}. Let $p: (X, \alpha: C \to X) \to (Y, \beta: C \to Y)$ be a $\mathcal{C}$-acyclic $\mathcal{C}$-fibration between $\mathcal{C}$-cofibrant objects. Then, working in $\textbf{G-Sp}$, we can solve  a lifting problem of $\beta$ against $p$ to find a map $s:(Y, \beta) \to (X, \alpha)$ in $(C \downarrow \textbf{G-Sp})$ with $ps=1$. Solving a lifting problem of $X \times \{0,1\} \cup C \times I \to X \times I$ against $p$, shows that $sp$ is homotopic to $1$ rel $p$. Now consider a lifting problem of $i$ against $p$ in $(C \downarrow \textbf{G-Sp})$. It suffices to find a solution to the underlying lifting problem in $\textbf{G-Sp}$ since any lift will automatically define a map in $(C \downarrow \textbf{G-Sp})$. We can solve this lifting problem as in Lemma \ref{hco1}, with the initial up to homotopy lift as $sw$.
\end{proof}

\begin{example} \label{hco5} 
	Suppose that $H \in \mathcal{C}$. If $U$ is an equivariant open subspace of $G/H \times D^n$, then $i: U \cap (G/H \times \partial D^n) \to U \cap (G/H \times D^n)$ is the inclusion of the boundary of a smooth $G$-manifold and so is an $hG$-cofibration between $\mathcal{C}$-cofibrant spaces by the equivariant collar neighbourhood theorem, \cite[Theorem 3.5]{K07}. Therefore, Lemma \ref{hco2} tells us that $i$ is a $\mathcal{C}$-cofibration. 
\end{example}

\begin{lemma} \label{hco10}
	If $i: A \to B$ is a $\mathcal{C}$-cofibration and $U \subset B$ is a $G$-equivariant open subspace, then $i: A \cap U \to U$ is a $\mathcal{C}$-cofibration.
\end{lemma}

\begin{proof}
	Firstly, $i$ is a retract of a relative $\mathcal{C}$-cell complex, $j: A \to C$, as shown below:
	
		\[	\begin{tikzcd}
		A  \arrow{d}{i} \arrow{r} & A \arrow{d}{j} \arrow{r} & A \arrow{d}{i} \\
		B \arrow[swap]{r}{i} & C \arrow[swap]{r}{r} & B 		
	\end{tikzcd} \]

If we let $V = r^{-1}(U)$, then $i: A \cap U \to U$ is a retract of $j: A \cap V \to V$. So we have reduced to the case where $i$ is a relative $\mathcal{C}$-cell complex, which we now assume. 

We can express $i$ as a transfinite composite of maps $A_{\lambda} \to A_{\lambda + 1}$, each of which is a pushout of a map of the form $G/H \times S^{n-1} \to G/H \times D^n$, with $H \in \mathcal{C}$. By \cite[Lemma 6.2.3]{R23} and Lemma \ref{hco9}, the map $i: A \cap U \to U$ can be expressed as a transfinite composite of the maps $A_{\lambda} \cap U \to A_{\lambda + 1} \cap U$, each of which is a pushout of a map of the form $(G/H \times S^{n-1}) \cap \alpha^{-1}(U) \to (G/H \times D^n) \cap\alpha^{-1}(U)$, where $\alpha$ is the inclusion of the cell into $B$. By Example \ref{hco5}, $A_{\lambda} \cap U \to A_{\lambda + 1} \cap U$ is therefore a $\mathcal{C}$-cofibration, and, hence, so is the transfinite composite $i: A \cap U \to U$.
\end{proof}

\begin{corollary} \label{hco11}
	An equivariant open subspace of a $\mathcal{C}$-cofibrant $G$-space is $\mathcal{C}$-cofibrant.
\end{corollary}

Combining Lemma \ref{hco2} and Corollary \ref{hco11}, we can rephrase Lemma \ref{hco2} as follows:

\begin{theorem} \label{hco6}
	If $i: A \to B$ is an $hG$-cofibration and $B$ is $\mathcal{C}$-cofibrant, then $A$ is $\mathcal{C}$-cofibrant and $i$ is a $\mathcal{C}$-cofibration.	
\end{theorem}

\begin{proof} If $(H, \lambda)$ represents $(B,A)$ as a  $G$-NDR-pair, then $A$ is a retract of $\lambda^{-1}([0,1))$ which is $\mathcal{C}$-cofibrant by Corollary \ref{hco11}. 	
	\end{proof}
	
	To close this section, we record the following straightforward lemma concerning the restriction of the group action to a subgroup. Let $\mathcal{C}$ be a set of subgroups of $G$ which is closed under conjugacy and let $H$ be a subgroup of $G$.  Define $\mathcal{C} \cap H = \{K \cap H | K \in \mathcal{C}\}$. We then have:
	
	\begin{lemma} \label{hco17}
		If $X$ is $\mathcal{C}$-cofibrant as a $G$-space, then $X$ is $(\mathcal{C} \cap H)$-cofibrant as an $H$-space.
	\end{lemma}

\begin{proof}
	We can reduce to the case where $X$ is a $\mathcal{C}$-cell complex via a retract argument. In this case, $X$ is a transfinite composite of pushouts of maps of the form $i: G/K \times S^{n-1} \to G/K \times D^n$ with $K \in \mathcal{C}$. Therefore, it suffices to show that $i$ is a $(\mathcal{C} \cap H)$-cofibration when viewed as an $H$-map. Indeed, $i$ is an $hH$-cofibration between $H$-manifolds with isotropy groups in $\mathcal{C} \cap H$, so is a $(\mathcal{C} \cap H)$-cofibration by \cite[Corollary 7.2]{I83} and Theorem \ref{hco6}.
\end{proof}

\section{A pushout-product property of symmetrizable cofibrations} \label{hco81}

Let $\mathcal{E}$ be a closed symmetric monoidal, tensored topological model category with monoidal product $\boxtimes$. If $f: A \to B$ and $g: C \to D$ are maps in $\mathcal{C}$, let $f \widehat{\boxtimes} g$ denote the pushout-product of $f$ and $g$, $f \widehat{\boxtimes} g: A \boxtimes D \cup_{A \boxtimes C} B \boxtimes C \to B \boxtimes D$. In general, a hat above an operation denotes either the corresponding pushout-product or the dual notion involving pullbacks. The $n$-fold composite $f \widehat{\boxtimes} ...  \widehat{\boxtimes} f$ is denoted by $f^{\widehat{\boxtimes} n}$. The notion of a symmetrizable cofibration is then defined as follows:

\begin{definition}
	A map $i: A \to B$ in $\mathcal{E}$ is said to be a symmetrizable (acyclic) cofibration if for every $n \geq 1$, $h^{\widehat{\boxtimes} n} / \Sigma_n$ is an (acyclic) cofibration.
\end{definition}
	
	\begin{example} \label{hco4}
		Let $G = \Sigma_n$ and let $\mathcal{C}$ be the set of subgroups of $G$, each isomorphic to a product of symmetric groups $\Sigma_{n_1} \times ... \times \Sigma_{n_l}$, corresponding to partitions of $\{1,...,n\}$ into sets of size $n_1, ..., n_l$. Let $i: S^{k-1} \to D^k$ be the boundary inclusion. If we consider the iterated pushout product map $j: S^{k-1} \times (D^k)^{n-1} \cup ... \cup (D^k)^{n-1} \times S^{k-1} := Q_n(i) \to (D^k)^n$, then $j$ is an $h \Sigma_n$-cofibration by \cite[Lemma A.4]{M72}, and the domain and codomain are $\mathcal{C}$-cofibrant since they are $G$-manifolds, hence $G$-CW complexes, \cite[Corollary 7.2]{I83}, with isotropy groups in $\mathcal{C}$. Therefore, Lemma $\ref{hco2}$ tells us that $j$ is a $\mathcal{C}$-cofibration.
	\end{example}
	
	We then have the following application to the theory of symmetrizable cofibrations, which is a generalisation of \cite[Proposition 2.1.12]{S18}:
	
	\begin{theorem}
		Let $j: A \to B$ be a symmetrizable cofibration in a tensored and cotensored closed symmetric monoidal topological model category and let $i: C \to D$ be a $q$-cofibration of spaces.  Then the pushout product map:	
		\[
		j \widehat{\times} i : A \times D \cup_{A \times C} B \times C \to B \times D
		\]
		
		is a symmetrizable cofibration, which is a symmetrizable acyclic cofibration if either $i$ is acyclic or $j$ is a symmetrizable acyclic cofibration.
	\end{theorem}
	
	\begin{proof}
		 The assumption that $\mathcal{E}$ is tensored and cotensored implies that the bifunctor $- \times -$ preserves colimits in both variables. It follows that we can reduce to the case where $i$ is a generating $q$-cofibration via \cite[Proposition C.2.9]{RV22}, and using the closure properties of symmetrizable cofibrations described in \cite[Theorem 7]{GG16}. Now, for any $n \geq 1$, we want to show that $(j \widehat{\times} i)^{\widehat{\boxtimes} n} / \Sigma_n$ is an (acyclic) cofibration. Equivalently, we want to show that $(j \widehat{\times} i)^{\widehat{\boxtimes} n}$ has the $\Sigma_n$-equivariant LLP with respect to (acyclic) fibrations $p: X \to Y$, where $\Sigma_n$ acts trivially on $X$ and $Y$. Using the tensor adjunction, it suffices to show that $i^{\widehat{\times} n}$ has the equivariant LLP wrt $\widehat{Map}(j^{\widehat{\boxtimes n}},p)$. In Example \ref{hco4}, we saw that $i^{\widehat{\times} n}$ is a $\mathcal{C}$-cofibration, with respect to the specified set of subgroups $\mathcal{C}$. So it suffices to show that $\widehat{Map}(j^{\widehat{\boxtimes n}},p)$ is a $\mathcal{C}$-acyclic $\mathcal{C}$-fibration. We need to show that for every $H \in \mathcal{C}$,   $\widehat{Map}(j^{\widehat{\boxtimes n}} / H,p)$ is a $q$-acyclic $q$-fibration. Since $H \in \mathcal{C}$ it is isomorphic to a group of the form $\Sigma_{n_1} \times ... \times \Sigma_{n_l}$, so $j^{\widehat{\boxtimes n}} / H$ is isomorphic to $(j^{\widehat{\boxtimes} n_1} / \Sigma_{n_1}) \widehat{\boxtimes} ... \widehat{\boxtimes} (j^{\widehat{\boxtimes}n_l} / \Sigma_{n_l})$. Therefore, the result follows from the fact that the model structure is topological and monoidal, $j$ is a symmetrizable cofibration and either $p$ is acyclic or $j$ is a symmetrizable acyclic cofibration.
	\end{proof}

\section{Local gluing results for fibrations and cofibrations} \label{hco82}

In this subsection, we explain how Theorem \ref{hco6} allows us to prove gluing results for $\mathcal{C}$-fibrations and $\mathcal{C}$-cofibrations analogous to Dold's gluing results for $h$-fibrations and $h$-cofibrations, \cite[Theorem 4.8]{D63} and \cite[Satz 2]{D68}. First recall that:

\begin{definition}
	A $G$-numerable open cover $\{U_j\}_{j \in J}$ of a $G$-space, $B$, is a locally finite open cover of $B$ such that for every $j \in J$ there exists an equivariant map $\lambda_j : B \to I$ with $U_j = \lambda_j^{-1}((0,1])$.
\end{definition}

The next lemma is central to the proofs of all four gluing theorems:

\begin{lemma} \label{hco19}
	Consider a lifting problem of $G$-spaces as below:	
	\begin{equation} \label{hco24}	
		\begin{tikzcd}
			A \arrow{r}{f} \arrow[swap]{d}{i} & X \arrow{d}{p} \\
			B \arrow[swap]{r}{g} & Y 
		\end{tikzcd}
	\end{equation}
	
	Suppose that there exists a $G$-numerable open cover $\{U_i\}_{i \in I}$ of $B$ such that if $i \in I$ and $J \subset I$, then there exists a solution to the induced lifting problem:
	
	\[
	\begin{tikzcd}
		i^{-1}(U_i \cap (\cup_{j \in J} U_j))  \arrow{r}{f} \arrow[swap]{d}{i} & X \arrow{d}{p} \\
		U_i \cap (\cup_{j \in J} U_j) \arrow[swap]{r}{g} & Y 
	\end{tikzcd}
	\]
	
	and given two solutions $\psi_1, \psi_2$ of the induced lifting problem there exists a $G$-homotopy from $\psi_1$ to $\psi_2$ which is relative to both $p$ and $i$.
	
	Then there is a solution to the original lifting problem \ref{hco24}.
\end{lemma}
\begin{proof}
	Let:
	\[ S = \{(A, \psi) | A \subset I, \psi: \cup_{\alpha \in A} U_\alpha \to X, p \psi = g, \psi i = f\}. \]
	
	We define a partial order on $S$ by:
	\[ (A, \psi_1) \leq (B, \psi_2)\]
	
	iff $A \subset B$ and $\psi_1 = \psi_2$ on $\cup_{\alpha \in A} U_{\alpha} \setminus \cup_{\beta \in B \setminus A} U_{\beta}$. 
	
	We may as well assume that $B$ is non-empty, in which case taking $J = \{i\}$ in the condition of the lemma shows that $S$ is non-empty. Since $\{U_i\}$ is a locally finite cover, every chain in $S$ has an upper bound. Hence, by Zorn's lemma, $S$ contains a maximal element, say $(A, \psi_1)$. If $A \neq I$, there exists some $\beta \in I \setminus A$. By the condition in the lemma, there exists some $(\{\beta\}, \psi_2) \in S$. By the uniqueness up to relative homotopy condition, there exists a $G$-homotopy $K: (U_{\beta} \cap (\cup_{\alpha \in A} U_{\alpha})) \times I \to X$ between $\psi_1$ and $\psi_2$, relative to $p$ and $i$. Pick $G$-maps $\lambda_1, \lambda_2: U_{\beta} \cup (\cup_{\alpha \in A} U_{\alpha}) \to I$ such that $\cup_{\alpha \in A} U_\alpha = \lambda_1^{-1}((0,1])$, $U_{\beta} = \lambda_2^{-1}((0,1])$ and $\lambda_1 + \lambda_2 = 1$. Define:
	
	\begin{equation*}
		\psi(u) = 
		\begin{cases}
			\psi_1(u) & \text{if} \ \lambda_1 > \frac{3}{4} \\
			\psi_2(u) & \text{if} \ \lambda_1 < \frac{1}{4} \\
			K(u, \frac{3}{2} - 2 \lambda_1) & \text{otherwise} 
		\end{cases}
	\end{equation*}
	
	Then $(A \cup \{\beta\}, \psi) \in S$ with $(A, \psi_1) \leq (A \cup \{\beta\}, \psi)$, a contradiction. So $A = I$ and $\psi_1$ is a solution to the original lifting problem.	
\end{proof}

As an immediate consequence, we have the following gluing result for $\mathcal{C}$-cofibrations:

\begin{theorem} \label{hco21}
	If $i: A \to B$ is a map of $G$-spaces and there is a $G$-numerable open cover $\mathcal{U}$ of $B$ such that for every $U \in \mathcal{U}$, $i|_{A \cap U}: A \cap U \to U$ is a $\mathcal{C}$-cofibration, then $i$ is a $\mathcal{C}$-cofibration.
\end{theorem}

\begin{proof}
	We consider a lifting problem for $i$ with respect to a $\mathcal{C}$-acyclic $\mathcal{C}$-fibration, $p: X \to Y$. The $G$-numerable open cover $\mathcal{U}$ ensures the existence of local lifts as in Lemma \ref{hco19}. Uniqueness follows from Lemma \ref{hco10} and the fact that the $\mathcal{C}$-model structure is topological, \cite[Proposition B.5]{S18}, so solutions to such lifting problems are unique up to relative homotopy. Therefore, the result follows from Lemma \ref{hco19}. 
\end{proof}

The corresponding gluing theorem for $h$-cofibrations is \cite[Satz 2]{D68}, whose proof uses the fact that the restriction of an $h$-cofibration to a numerable open subspace is an $h$-cofibration, \cite[Satz 1]{D68}, in place of Lemma \ref{hco10}:

\begin{theorem} \label{hco57}
	Suppose that $\{U_j\}_{j \in J}$ is a $G$-numerable open cover of $B$. If $f:A \to B$ is a map such that $f: f^{-1}(U_j) \to U_j$ is an $hG$-cofibration for all $j \in J$, then $f:A \to B$ is an $hG$-cofibration.
\end{theorem}

As for $h$-fibrations, \cite[Theorem 4.8]{D63}, with a little more work we can use Lemma \ref{hco19} to prove a gluing result for $\mathcal{C}$-fibrations. In fact, as in the non-equivariant setting, \cite[Theorem 6.3.3]{D08}, we could use a more direct argument using cubical subdivisions of $I^n$ to prove the theorem below. However, the proof we give works in exactly the same way in the $hG$-fibration case, provided we assume that the open cover is $G$-numerable:

\begin{theorem} \label{hco22}
	If $p: X \to Y$ is a map and there exists a $G$-equivariant open cover $\{U_i\}_{i \in I}$ of $Y$ such that $p: p^{-1}(U_i) \to U_i$ is a $\mathcal{C}$-fibration for all $i \in I$, then $p$ is a $\mathcal{C}$-fibration.
\end{theorem}

\begin{proof}
	
	Consider a lifting problem:
	
	\[
	\begin{tikzcd}
		A \arrow{r}{f} \arrow[swap]{d}{i} & X \arrow{d}{p} \\
		A \times I \arrow[swap]{r}{H} & Y 
	\end{tikzcd}
	\]
	
	where $A$ is $\mathcal{C}$-cofibrant and $i: A \to A \times I$ is the inclusion of $A \times \{0\}$. If $Y$ is not paracompact Hausdorff then we can pull the lifting problem back along $H$ to reduce to the case of a lifting problem involving $i$ and $q$, where $q$ is the pullback of $p$. Note that there is an open cover $\{V_i := H^{-1}(U_i)\}$ of $A \times I$ such that  $q: q^{-1}(U_i) \to U_i$ is a $\mathcal{C}$-fibration for all $i \in I$. Moreover, the base of $q$ is paracompact Hausdorff, so $\{V_i\}$ has a $G$-numerable refinement. Therefore, we assume from now on that the open cover $\{U_i\}$ is $G$-numerable.
	
	Choose functions, $\lambda_i: Y \to I$, for every $i \in I$, such that $U_i = \lambda_i^{-1}((0,1])$ and $\sum_{i \in I} \lambda_i = 1$. Suppose that $S = (i_0,...,i_{n-1})$ is a finite sequence of elements of $I$, so, in particular, the members of the sequence need not be distinct. Define a function $\mu_S: A \to I$ by:
	
	\begin{equation} \label{hco23} 
		\mu_S(a) = \max \{\min_{j \in \{0...,n-1\}} \left(\inf_{t \in [\frac{j}{n}, \frac{j+1}{n}]} \lambda_{i_j}(H(a,t)) \right) - n \sum_{T| |T| < n} \mu_T, 0\}
	\end{equation}
	
	Then $\{V_S := \mu_S^{-1}((0,1])\}$ is a $G$-numerable open cover of $A$, indexed over finite sequences of elements of $I$. Therefore, $\{V_S \times I\}$ is a $G$-numerable open cover of $A \times I$. Moreover, by inspection of equation (\ref{hco23}), we see that $H(V_S \times [\frac{j}{n}, \frac{j+1}{n}]) \subset U_{i_j}$ for every $j$. It follows that the existence condition of Lemma \ref{hco19} is satisfied, since $V_S$ is $\mathcal{C}$-cofibrant by Corollary \ref{hco11} and so we can construct a lift inductively on these intervals. Similarly, we can construct relative homotopies between any two lifts inductively, and so the uniqueness condition of Lemma \ref{hco19} is also satisfied, which gives the result.
\end{proof}

Finally, we record the corresponding gluing result for $hG$-fibrations, which follows as above or as in \cite[Theorem 4.8]{D63}:

\begin{theorem}
	If $p: X \to Y$ is a map and there exists a $G$-numerable open cover $\{U_i\}_{i \in I}$ of $Y$ such that $p: p^{-1}(U_i) \to U_i$ is an $hG$-fibration for all $i \in I$, then $p$ is an $hG$-fibration.
\end{theorem}

\section{Applications to $n$-fold products and the diagonal map} \label{hco83}

Recall that the wreath product $G \wr \Sigma_n$ is defined to be the semidirect product $G^n \rtimes \Sigma_n$ induced by the permutation action of $\Sigma_n$ on $G^n$.
\begin{definition}
	Let $\mathcal{D}_n$ be the set of isotropy groups of elements of $I^n$ under the action of the symmetric group, $\Sigma_n$, given by coordinate permutation. Let $\mathcal{C}$ be a set of subgroups of $G$ closed under conjugacy. Define $\mathcal{C} \wr \mathcal{D}_n$ to be the conjugate closure of the set of subgroups of $G \wr \Sigma_n$ of the form $(H_1 \times ... \times H_n) \rtimes Q$, where $Q \in \mathcal{D}_n$, each $H_i \in \mathcal{C}$ and, for all $i$, if $\sigma \in Q$ then $H_i = H_{\sigma(i)}$.
\end{definition}

\begin{lemma} \label{hco59}
	If $X$ is a $\mathcal{C}$-cofibrant $G$-space, then $X^n$ is $(\mathcal{C} \wr \mathcal{D}_n)$-cofibrant as a $(G \wr \Sigma_n)$-space.
\end{lemma}

\begin{proof}
	Since $X$ is a retract of a $\mathcal{C}$-cell complex, we can reduce to the case where $X$ is a $\mathcal{C}$-cell complex. For the purposes of this proof, define a $\mathcal{C}$-pseudocomplex to be a complex built out of pushouts of maps of the form $A \times S^{n-1} \to A \times D^n$ where $A$ is a finite product of $G$-spaces of the form $G / H$ with $H \in \mathcal{C}$. In particular, if $X$ is a $\mathcal{C}$-cell complex, then $X^n$ inherits the structure of a $\mathcal{C}$-pseudocomplex. Let $\mathcal{E}$ denote the category of $(\mathcal{C} \wr \mathcal{D}_n)$-cofibrant $\Sigma_n$-equivariant $\mathcal{C}$-subpseudocomplexes of $X^n$ and inclusions which are $(\mathcal{C} \wr \mathcal{D}_n)$-cofibrations. Let $S_{\lambda}$ denote the set of $\lambda$-sequences in $\mathcal{E}$ such that none of the successor maps are identity cofibrations, \cite[Definition 10.2.1]{H03}, and let $S = \cup{S_\lambda}$ which is still a set due to the restriction on identity maps. We can define a partial order on $S$ by $F \leq G$ iff $F = G$ or $F$ is the restriction of $G$ to some lesser ordinal. Then $S$ is non-empty and every chain has an upper bound, so $S$ has a maximal element by Zorn's lemma, say $F: \lambda \to \mathcal{E}$. Note that $\lambda$ will be a successor ordinal since otherwise we could take a union which contradicts maximality. 
	
	Suppose that $M := F_{\lambda - 1} \neq X^{n-1}$. Then we can attach a $\mathcal{C}$-pseudocell to $M$ along an attaching map $\alpha: G/H_1 \times ... \times G/H_n \times S^{m-1} \to M$, corresponding to a pushout-product of $n$-cells from $X$. Now $\Sigma_n$ acts on $n$-tuples of cells of $X$, and therefore on the orbit $\Sigma_n / K$ of the $n$-tuple inducing $\alpha$, where $K$ is the stabiliser of the $n$-tuple inducing $\alpha$. Since $M$ is $\Sigma_n$-equivariant, none of the permuted pseudocells are contained in $M$, so the pushout below defines a $\Sigma_n$-equivariant $\mathcal{C}$-subpseudocomplex, $M^{'}$, of $X^n$: 
	
		\[
	\begin{tikzcd}
		\sqcup_{\sigma \in \Sigma_n / K} G / H_{\sigma^{-1}(1)} \times ... \times G / H_{\sigma^{-1}(n)} \times S^{m-1} \arrow{r}{} \arrow[swap]{d}{i} & M \arrow{d}{} \\
		\sqcup_{\sigma \in \Sigma_n / K} G / H_{\sigma^{-1}(1)} \times ... \times G / H_{\sigma^{-1}(n)} \times D^m \arrow[swap]{r}{} & M^{'} 
	\end{tikzcd}
	\]
	
	Since the domain and codomain of $i$ are both $(G \wr \Sigma_n)$-manifolds with isotropy groups in $\mathcal{C} \wr \mathcal{D}_n$, they are $(\mathcal{C} \wr \mathcal{D}_n)$-cofibrant by \cite[Corollary 7.2]{I83}. Moreover $i$ is an $h(G \wr \Sigma_n)$-cofibration since it is the product of the $(G \wr \Sigma_n)$-space $\sqcup_{\sigma \in \Sigma_n / K} G/H_{\sigma^{-1}(1)} \times ... \times G/H_{\sigma^{-1}(n)}$ with the $h(G \wr \Sigma_n)$-cofibration $S^{m-1} \to D^m$, which is an $h(G \wr \Sigma_n)$-cofibration by \cite[Lemma A.4]{M72} and the fact that $G$ acts trivially on $S^{m-1}$ and $D^m$. Therefore, $i$ is a $(\mathcal{C} \wr \mathcal{D}_n)$-cofibration by Theorem \ref{hco6}, so $M \to M^{'}$ is a $(\mathcal{C} \wr \mathcal{D}_n)$-cofibration, a contradiction. So, in fact, $X^{n}$ is $(\mathcal{C} \wr \mathcal{D}_n)$-cofibrant.
\end{proof}

If we forget the symmetric action and allow $G$ to act diagonally on $X^n$, then Lemma \ref{hco17} has the following consequence:

 \begin{corollary}
 	If $\mathcal{C}$ is closed under conjugacy and finite intersections and $X$ is a $\mathcal{C}$-cofibrant $G$-space, then $X^n$ is $\mathcal{C}$-cofibrant as a $G$-space.
 \end{corollary}

Before moving on to consider the diagonal map, we record the following lemma which we will need later on:

\begin{lemma} \label{hco58}
	If $M$ is a smooth $G$-manifold with boundary, then the inclusion, $i$, of the union of the (topological) boundary of $M^n$ and the diagonal $\Delta_M$ into $M^n$ is an $h(G \times \Sigma_n)$-cofibration.
\end{lemma}

\begin{proof}
	If $U$ is the interior of $M$, then the diagonal map $\Delta: U \to U^n$ is an $h(G \times \Sigma_n)$-cofibration by \cite[Theorem 4.4, Lemma 4.2]{K07}. Similarly, $\partial M \to M$ is an $hG$-cofibration by \cite[Theorem 3.5]{K07} and so $\partial(M^n) \to M^n$ is an $h(G \times \Sigma_n)$-cofibration by Lemma \ref{hco56}. It follows that for any point, $x$, not in the intersection of $\Delta_M$ and $\partial(M^n)$, there is a $(G \times \Sigma_n)$-equivariant open neighbourhood $V$ of $x$ such that $i \cap V$ is an $h(G \times \Sigma_n)$-cofibration, by \cite[Satz 1]{D68}. Since $M^n$ is paracompact Hausdorff any $(G \times \Sigma_n)$-equivariant open cover has a $(G \times \Sigma_n)$-numerable refinement, so by Theorem \ref{hco57} it suffices to find such $V$ around points in $\partial(M^n) \cap \Delta_M$. By \cite[Theorem 3.5]{K07}, we can reduce to the case where $M = \partial M \times [0,1)$. In this case $i$ is the composite $\partial (M^n) \cup \Delta_M \xrightarrow{j} \partial(M^n) \cup ((\partial M)^n \times \Delta_{[0,1)}) \xrightarrow{k} (\partial M)^n \times [0,1)^n$. Note that $j$ is a pushout of the pushout-product of the maps $\Delta_{\partial M} \to (\partial M)^n$ and $\{0\} \to \Delta_{[0,1)}$. Therefore, $j$ is an $h(G \times \Sigma_n)$-cofibration by \cite[Theorem 4.4, Lemma 4.2]{K07}. Moreover, $k$ is the product of $(\partial M)^n$ with $\partial([0,1)^n) \cup \Delta_{[0,1)} \xrightarrow{\alpha} [0,1)^n$, so we just need to show that $\alpha$ is an $h \Sigma_n$-cofibration. This follows from the fact that $\partial((\Delta^1)^n) \cup \Delta_{\Delta^1}$ is a $\Sigma_n$-equivariant subset of the simplicial set $(\Delta^1)^n$ with the permutation action of $\Sigma_n$, so upon passage to geometric realisation, and restriction to a $\Sigma_n$-numerable open subset, we obtain that $\alpha$ is an $h(\Sigma_n)$-cofibration as desired.
\end{proof}

Now consider the diagonal map $\Delta: X \to X^n$. Since $G \wr \Sigma_n$ does not act on $X$, we will not be able to show that $\Delta$ is an $h(\mathcal{C} \wr \mathcal{D}_n)$-cofibration. Instead we consider the $(G \times \Sigma_n)$-action induced by restriction along the diagonal map $G \to G^n$, in which case we have the following equivariant analogue of \cite[Corollary III.2]{DE72}:

\begin{theorem} \label{hco15}	If $X$ is $\mathcal{C}$-cofibrant, then $\Delta : X \to X^n$ is a $((\mathcal{C} \wr \mathcal{D}_n) \cap (G \times \Sigma_n))$-cofibration, see Lemma \ref{hco17}. In particular, if $X$ is $qG$-cofibrant, then $\Delta$ is a $q(G \times \Sigma_n)$-cofibration.
\end{theorem}

\begin{proof}
We know that $X^n$ is $((\mathcal{C} \wr \mathcal{D}_n) \cap (G \times \Sigma_n))$-cofibrant by Lemma \ref{hco59} and Lemma \ref{hco17}. Since any $\mathcal{C}$-cofibrant $G$-space is a retract of a $\mathcal{C}$-cell complex, by Lemma \ref{hco2} it suffices to show that if $Z$ is a $\mathcal{C}$-cell complex then $\Delta : Z \to Z^n$ is an $h(G \times \Sigma_n)$-cofibration. To this end, assume that $\Delta: Y \to Y^n$ is an $h(G \times \Sigma_n)$-cofibration and $X$ is obtained from $Y$ by attaching a single cell along $\alpha: G/H \times S^{k-1} \to Y$. We will show that if $(H, \lambda)$ represents $(Y^n, Y)$ a $(G \times \Sigma_n)$-NDR-pair, then there is a $(G \times \Sigma_n)$-NDR pair $(K, \mu)$ for $(X^n, X)$ such that $K|_{Y^n \times I} = H$ and $\mu|_{Y^n} = \lambda$. Firstly, the composite $Y \to Y^n \to Y \times X^{n-1} \cup ... \cup X^{n-1} \times Y$ is an $h(G \times \Sigma_n)$-cofibration, which can be represented by a $(G \times \Sigma_n)$-NDR pair that extends $(H, \lambda)$ by Lemma \ref{hco51}, where note $Y^n \to Y \times X^{n-1} \cup ... \cup X^{n-1} \times Y$ is an $h(G \times \Sigma_n)$-cofibration by Lemma \ref{hco56} and the equivariant analogue of \cite[Lemma 5]{S72}. We then consider the map between pushouts:
	
\[	\begin{tikzcd}
		G/H \times D^k  \arrow{d}{\Delta} & G/H \times S^{k-1} \arrow{l} \arrow{d}{\Delta} \arrow{r} & Y \arrow[tail]{d}{\alpha} \\
		(G/H)^n \times D^{kn} & (G/H)^n \times \partial D^{kn} \arrow{l} \arrow{r} & Y \times X^{n-1} \cup ... \cup  X^{n-1} \times Y
	\end{tikzcd} \]

We've seen that $\alpha$ is an $h(G \times \Sigma_n)$-cofibration. The map between pushouts is $\Delta : X \to X^n$ and to show this is an $h(G \times \Sigma_n)$-cofibration it suffices to show the pushout-product map, $\phi$, below is an $h(G \times \Sigma_n)$-cofibration:
\[
\phi: G/H \times D^k \cup_{G/H \times S^{k-1}} (G/H)^n \times \partial D^{kn} \to (G/H)^n \times D^{kn} 
\]

This follows from Lemma \ref{hco58}, and so $\Delta$ is an $h(G \times \Sigma_n)$-cofibration, being the composite of a pushout of $\alpha$ and a pushout of $\phi$. Moreover, the $h(G \times \Sigma_n)$-cofibration $X \to X \times X$ can be represented by a $(G \times \Sigma_n)$-NDR-pair that extends $(H, \lambda)$ by Lemmas \ref{hco51} and \ref{hco52}. The lemma now follows by passage to colimits.
\end{proof}

\begin{corollary}
	If $\mathcal{C}$ is closed under conjugacy and finite intersections and $X$ is $\mathcal{C}$-cofibrant, then $\Delta: X \to X^n$ is a $\mathcal{C}$-cofibration.
\end{corollary}

\section{A $q$-fibration between $q$-cofibrant spaces is an $h$-fibration} \label{hco84}

As in \cite{DE72}, we make the following definition:

\begin{definition} \label{hco12}
	We call a $G$-space, $X$, $G$-equivariantly locally equiconnected (or $G$-LEC) if $\Delta: X \to X \times X$ is an $hG$-cofibration.
\end{definition}

The next lemma is the equivariant analogue of \cite[Theorem II.3]{DE72}:

\begin{lemma} \label{hco13}
	Let $X$ be $G$-LEC and suppose that $x \in X$ has isotropy group $H$. Then there exists an $H$-equivariant open neighbourhood $U$ of $x$ such that $\pi_1, \pi_2:  U \times U \to X$ are $H$-homotopic. Here, $\pi_i$ denotes the projection onto the $i$th factor followed by the inclusion into $X$.
\end{lemma}

\begin{proof}
	Let $(H, \lambda)$ represent the inclusion of the diagonal $(X \times X, \Delta)$ as a $G$-NDR pair. Identify $X$ with $\{x\} \times X$ and let $U = \lambda^{-1}([0,1)) \cap (\{x\} \times X)$. Consider $H: \{x\} \times U \times I \to X \times X$. Then $\pi_2 H$ defines an $H$-homotopy from the inclusion $i: U \to X$ to some map $f$. Moreover, $\pi_1 H$ defines an $H$-homotopy from $f$ to the constant map to $x$. If we now consider $\pi_1, \pi_2 :  U \times U \to X$, we see that $\pi_1$ is the composite $U \times U \to X \times U \xrightarrow{\pi_1} X$, which is $H$-homotopic to the constant map to $x$. Similarly, $\pi_2$ is $H$-homotopic to the constant map to $x$, so $\pi_1$ and $\pi_2$ are $H$-homotopic.
\end{proof}

\begin{lemma} \label{hco14}
	Let $X$ be $\mathcal{C}$-cofibrant and suppose that $x \in X$ has isotropy group $H$. Then there exists an $H$-equivariant open neighbourhood $U$ of $x$ such that $\pi_1, \pi_2: U \times U \to X$ are $H$-homotopic via an $H$-homotopy which is constant on $\Delta$.
\end{lemma}

\begin{proof}
	By Lemma \ref{hco15} and Lemma \ref{hco13}, we know that $\pi_1$ and $\pi_2$ are $H$-homotopic, say via an $H$-homotopy $L:  U \times U \times I \to X$. Define $\Gamma: \Delta \times I \to X$ to be the restriction of $L$ to $\Delta \times I$. Let $\Gamma^{-1}$ denote the inverse homotopy from $\pi_2$ to $\pi_1$. Define $K$ to be composite (via concatenation) homotopy of $L$ and $\Gamma^{-1} \pi_2$:	
	\[  U \times U \times I \xrightarrow{\pi_2}  U \times I \cong  \Delta \times I \xrightarrow{\Gamma^{-1}} X	
\]

Then the restriction of $K$ to $\Delta \times I$ is the composite of $\Gamma$ and $\Gamma^{-1}$, and $K$ is an $H$-homotopy from $\pi_1$ to $\pi_2$. The composite of $\Gamma$ and $\Gamma^{-1}$ is itself $H$-homotopic, say via $\alpha: \Delta \times I \times I \to X$, to the constant homotopy on $\pi_1 |_\Delta$, and $\alpha$ can be made relative to $ \Delta \times \{0,1\} \times I$. We can then define a lifting problem:

	\[	\begin{tikzcd}
 U \times U \times I \times \{0\} \cup \Delta \times I \times I \cup U \times U \times \{0,1\} \times I \arrow{d} \arrow{rrrr}{K \cup \alpha \cup cnst(\{\pi_1, \pi_2\})} & & & & X \arrow{d} \\
 U \times U \times I \times I \arrow[swap]{rrrr} & & & & *		
\end{tikzcd} \]

A lift $\tilde{K}: U \times U \times I \times I \to X$ exists since $U$ is $(\mathcal{C} \cap H)$-cofibrant, by Corollary \ref{hco11} and Lemma \ref{hco17}, and, therefore, $H$-LEC by Lemma \ref{hco15}. This implies the left hand vertical map is an $hH$-acyclic $hH$-cofibration. Restricting $\tilde{K}$ to $ U \times U \times I \times \{1\}$ gives the desired $H$-homotopy.	
	\end{proof}

We will use the following consequence of \cite[Lemma 2.8]{L82}, which was noted in \cite[Lemma 6.1]{BF15}:

\begin{lemma} \label{hco18}
	If $p: E \to B$ is an $hH$-fibration, then $1 \times_H p: G \times_H E \to G \times_H B$ is an $hG$-fibration.
\end{lemma}

\begin{proof}
	Consider a lifting problem of $G$-spaces as in the right hand square below:	
\[
\begin{tikzcd}
G \times_H V_0 \arrow{r}{\cong} \arrow{d} &	X \arrow{r}{f} \arrow{d} & G \times_H E \arrow{d} \\
G \times_H (V_0 \times I) \arrow[swap]{r}{\cong} &	X \times I \arrow[swap]{r}{\alpha} & G \times_H B 
\end{tikzcd}
\]

Then $X \times I \cong G \times_H \alpha^{-1}(B)$ and so \cite[Lemma 2.8]{L82} tells us that there is an $H$-homeomorphism $\psi: V_0 \times I \to V$, where $V = \alpha^{-1}(B)$, $V_0 = V \cap (X \times \{0\})$ and $\psi|_{V_0 \times \{0\}}$ is the inclusion of $V_0$ into $V$. Note that $G \times_H V_0 \cong X$. The $H$-homeomorphism, $\psi$, induces a $G$-homeomorphism $G \times_H (V_0 \times I) \to X \times I$, as in the diagram above, which restricts at $0 \in I$ to a homeomorphism onto $X \times \{0\}$. So to solve the lifting problem of the right hand square, it suffices to solve the adjoint lifting problem of $H$-spaces:

\[
\begin{tikzcd}
	V_0 \arrow{r}{f} \arrow{d} & G \times_H E \arrow{d} \\
		V_0 \times I \arrow[swap]{r}{\alpha \psi} & G \times_H B 
\end{tikzcd}
\]

Since $V = \alpha^{-1}(B)$, $\alpha \psi: V_0 \times I \to G \times_H B$ factors through $B$, and $f: V_0 \to G \times_H E$ factors through $E$. Since $p$ is an $hH$-fibration, it follows that the lifting problem can be solved.
\end{proof}

A version of the following theorem was first proved in an equivariant context by Gevorgyan and Jimenez in \cite[Theorem 6]{GJ19}, using the $qG$-model structure and under the assumption that both $E$ and $B$ are $G$-CW complexes. The proof below corrects a minor gap in that proof, since it is not necessarily true that a $G$-CW complex, $X$, can be covered by $G$-equivariant open subspaces which are non-equivariantly contractible in $X$.

\begin{theorem} \label{hco16}
Suppose that $\mathcal{C}$ is closed under conjugacy and intersections.	If $p: E \to B$ is a $\mathcal{C}$-fibration between $\mathcal{C}$-cofibrant spaces, then $p$ is an $hG$-fibration.
\end{theorem}

\begin{proof}
	Let $(H, \lambda)$ represent $(B \times B, \Delta_B)$ as a $G$-NDR pair. As in Lemma \ref{hco14}, if $b \in B$ has isotropy group $H$ there is an $H$-equivariant open neighbourhood $U$ of $b$ and an $H$-homotopy $K: U \times U \times I \to B$ between $\pi_1$ and $\pi_2$ such that $K(u,v,t) = K(u,v,1)$ for all $t \geq \lambda(u,v)$. Note that $H \in \mathcal{C}$ since there exists an inclusion of $B$ into a $\mathcal{C}$-cell complex. Since $p: E \to B$ is a $(\mathcal{C} \cap H)$-fibration, $\mathcal{C}$ is closed under intersections and both $U$ and $p^{-1}(U)$ are $(\mathcal{C} \cap H)$-cofibrant, we can solve the lifting problem:
	
	\[
	\begin{tikzcd}
		U \times p^{-1}(U) \arrow{rr}{\pi_2} \arrow{d} & & E \arrow{d}{p} \\
		U \times p^{-1}(U) \times I  \arrow[swap]{rr}{K \circ (1 \times p \times 1)} & & B 
	\end{tikzcd}
	\]
	
	to define a slicing function $\tilde{K}: U \times p^{-1}(U) \times I \to E$. We now show that $p: p^{-1}(U) \to U$ is an $hH$-fibration. Indeed, given a lifting problem of $H$-spaces:
	
	\[
	\begin{tikzcd}
		X \arrow{r}{f} \arrow{d} & p^{-1}(U) \arrow{d}{p} \\
		X \times I \arrow[swap]{r}{L} & U 
	\end{tikzcd}
	\]
	
	a lift can be defined via $\tilde{L}(x,t) = \tilde{K}(L(x,t) ,f(x), \lambda(L(x,t), pf(x)))$.
	
	Since $B$ is completely regular, $b$ also has a $G$-equivariant	open neighbourhood $G$-homeomorphic to $G \times_H S$, where $S$ is an $H$-equivariant subspace of $B$ containing $b$, by \cite[Theorem 5.4]{B72}. Replacing $S$ by $S \cap U$ if necessary, we may as well assume $S \subset U$. In this case, it follows from the above that $p: p^{-1}(S) \to S$ is an $hH$-fibration and so $p: G \times_H p^{-1}(S) \to G \times_H S$ is an $hG$-fibration by Lemma \ref{hco18}. It follows that for every $b \in B$ there is a $G$-equivariant neighbourhood, $V$, of $b$ for which $p: p^{-1}(V) \to V$ is an $hG$-fibration. Since $B$ is paracompact, it follows that $p$ is an $hG$-fibration.
\end{proof}

\section{An alternative proof of Waner's Theorem} \label{hco85}

So far we have considered the $\mathcal{C}$-model structure and the $hG$-model structure on the category of $G$-spaces - however, we now have reason to consider the $m \mathcal{C}$-model structure, obtained by mixing the $\mathcal{C}$ and $hG$ model structures as in \cite[Theorem 2.1]{C06}. By \cite[Corollary 3.7]{C06}, a $G$-space is $m \mathcal{C}$-cofibrant iff it has the $G$-homotopy type of a $\mathcal{C}$-cofibrant $G$-space. Note that if $X$ is $\mathcal{C}$-cofibrant then just using subgroups $H \in \mathcal{C}$ in the proof of the $G$-CW approximation theorem, \cite[Ch. I, Theorem 3.6]{M96}, we obtain a $\mathcal{C}$-CW approximation $\gamma: Z \to X$ where $\gamma$ is a $G$-homotopy equivalence since a $\mathcal{C}$-equivalence between $\mathcal{C}$-cofibrant $G$-spaces is a $G$-homotopy equivalence. Therefore, a $G$-space is $m \mathcal{C}$-cofibrant iff it has the $G$-homotopy type of a $\mathcal{C}$-CW complex. Note that we are referring to unbased homotopy types throughout.

Using the results of the previous sections, we can give an alternative proof of Waner's theorem on $G$-spaces of the $G$-homotopy type of a $G$-CW complex, \cite[Corollary 4.14]{W80}:

\begin{theorem} \label{hco54}
Let $\mathcal{C}$ be closed under conjugacy and intersections. Suppose that $p: E \to B$ is an $hG$-fibration and $* \in B$ is a $G$-fixed basepoint. If $E$ and $B$ are $m \mathcal{C}$-cofibrant, then $F := p^{-1}(*)$ is $m \mathcal{C}$-cofibrant.
\end{theorem}

\begin{proof}
	The $\mathcal{C}$-model structure induces a model structure on the category $(\textbf{G-Sp}) \downarrow B$, \cite[Theorem 7.6.5]{H03}. Let $\tilde{p}$ be a fibrant cofibrant approximation to $p$ in this model category, \cite[Definition 8.1.22]{H03}. We can write $\tilde{p} = qi$, where $i$ is a trivial cofibration and $q$ is a fibration in $(\textbf{G-Sp}) \downarrow B$. Since $i$ has fibrant domain, since $p$ is also a $\mathcal{C}$-fibration, there is a weak equivalence $r$ such that $ri = 1$. Since we are working in $(G-Sp) \downarrow B$, this construction yields a commutative square:
	
\[	\begin{tikzcd}
		\tilde{E} \arrow[two heads]{r}{q} \arrow[swap, "\simeq" ']{d}{fr} & \tilde{B} \arrow[two heads, "\simeq" ']{d}{g} \\
		E \arrow[two heads, swap]{r}{p} & B
	\end{tikzcd} \]

Since $\tilde{E}$ and $\tilde{B}$ are $\mathcal{C}$-cofibrant, $q$ is an $hG$-fibration by Theorem \ref{hco16}. We know that $p$ is an $hG$-fibration by assumption. The vertical maps are $hG$-equivalences since $E$ and $B$ are $m \mathcal{C}$-cofibrant. Therefore, $\tilde{B}$ contains a $G$-fixed point and so $G \in \mathcal{C}$, since $\tilde{B}$ is a retract of a $\mathcal{C}$-cell complex. Since $g$ is a $\mathcal{C}$-acyclic $\mathcal{C}$-fibration, we can choose a $G$-fixed basepoint for $\tilde{B}$ which maps to $*$ under $g$. We also denote this basepoint by $*$. Since the $hG$-model structure is right proper, the induced map $q^{-1}(*) \to F$ is a $G$-homotopy equivalence. Since $* \to \tilde{B}$ is an $hG$-cofibration and $q$ is an $hG$-fibration, $q^{-1}(*) \to \tilde{E}$ is an $hG$-cofibration by \cite[Lemma 1.3.1]{MP12}. Since $\tilde{E}$ is $ \mathcal{C}$-cofibrant, it follows that $q^{-1}(*)$ is $\mathcal{C}$-cofibrant by Theorem \ref{hco6}. Therefore, $F$ is $m \mathcal{C}$-cofibrant as desired.
\end{proof}

Before moving on to the converse of Waner's theorem, we record some routine lemmas:

\begin{lemma} \label{hco60}
	If $X$ is an $m \mathcal{C}$-cofibrant $G$-space and $H$ is a subgroup of $G$, then $X$ is $m (H \cap \mathcal{C})$-cofibrant as an $H$-space.
\end{lemma}

\begin{proof}
	We know that $X$ is $G$-homotopy equivalent to a $\mathcal{C}$-cell complex, $Y$, which in turn is $(H \cap \mathcal{C})$-cofibrant as an $H$-space by Lemma \ref{hco17}.
\end{proof}

The next lemma follows as in \cite[Lemma 2.4]{M75}:

\begin{lemma} \label{hco61}
	If $p: E \to B \times I$ is an $hG$-fibration, then there is a $G$-homotopy equivalence over $B$, $E_0 \simeq E_1$, where $E_t$ is the preimage of $B \times \{t\}$.
\end{lemma}

\begin{corollary} \label{hco62}
	If $p: E \to \frac{G}{H} \times D^n$ is an $hG$-fibration and $* \in D^n$, then $E$ is $G$-homotopy equivalent over $\frac{G}{H} \times D^n$ to $(G \times_H F) \times D^n$, where $F = p^{-1}(1,*)$.
\end{corollary}

\begin{proof}
	Consider the pullback of $p$ along a $G$-homotopy $G/H \times D^n \times I \to G/H \times D^n$ between the identity map and the map induced by $D^n \to * \to D^n$. By Lemma \ref{hco61} and \cite[Lemmas 3.1 and 3.2]{R25}, $E$ is $G$-homotopy equivalent to $(G \times_H F) \times D^n$ over $\frac{G}{H} \times D^n$. 
\end{proof}

We will also need the following lemma, from which it follows that if $X$ is an $m \mathcal{D}$-cofibrant $H$-space, then $G \times_H X$ is $m\bar{\mathcal{D}}$-cofibrant as a $G$-space:

\begin{lemma} \label{hco63}
	If $H$ is a subgroup of $G$ and $X$ is a $\mathcal{D}$-cofibrant $H$-space, then $G \times_H X$ is $\bar{\mathcal{D}}$-cofibrant as a $G$-space, where $\bar{\mathcal{D}}$ is the conjugate closure of $\mathcal{D}$ viewed as a set of subgroups of $G$.
	\end{lemma}
\begin{proof}
	We may as well assume $X$ is a $\mathcal{D}$-cell complex, in which case it is a transfinite composite of pushouts of maps of the form $i: H/K \times S^{n-1} \to H/K \times D^n$ with $K \in \mathcal{D}$. Since $G \times_H - $ preserves colimits, it suffices to show that $G \times_H i$ is a $\bar{\mathcal{D}}$-cofibration. This follows from the observation that $G \times_H (H/K) \cong G/K$. 
\end{proof}

We can now prove the following general form of Waner's theorem and its converse:

\begin{theorem} \label{hco64}
	Let $\mathcal{C}$ be closed under conjugacy and intersections. Let $p: E \to B$ be an $hG$-fibration such that $B$ is $m \mathcal{C}$-cofibrant. Then $E$ is $m \mathcal{C}$-cofibrant iff for every $H$ and $b \in B$ with isotropy group $H$, $p^{-1}(b)$ is $m (H \cap \mathcal{C})$-cofibrant. \end{theorem}
\begin{proof}	
	($\implies$) If $E$ is $m \mathcal{C}$-cofibrant and $b \in B$ has isotropy group $H$, then both $E$ and $B$ are $m (H \cap \mathcal{C})$-cofibrant as $H$-spaces, by Lemma \ref{hco60}. Moreover, $p$ is an $hH$-fibration when viewed as an $H$-map, so $p^{-1}(b)$ is $m (H \cap \mathcal{C})$-cofibrant by Theorem \ref{hco54}.

	($\impliedby$) We'll first show that we can reduce to the case where $B$ is a $\mathcal{C}$-CW complex. Let $h: \tilde{B} \to B$ be a $\mathcal{C}$-CW approximation and consider the pullback:	
	\[	\begin{tikzcd}
		\tilde{E} \arrow{r}{k} \arrow[swap, two heads]{d}{q} & E \arrow[two heads]{d}{p} \\
		\tilde{B} \arrow[swap]{r}{h} & B 
	\end{tikzcd}
	\]	
	Since $B$ is $m \mathcal{C}$-cofibrant and the $hG$-model structure is proper, \cite[Theorem 6.1.2]{R23}, both $h$ and $k$ are $G$-homotopy equivalences. If $\tilde{b} \in \tilde{B}$, $b = h(\tilde{b})$, $\tilde{b}$ has isotropy group $\tilde{H}$, and $b$ has isotropy group $H$, then $H \supseteq \tilde{H}$. By assumption, $p^{-1}(b)$ is $m (H \cap \mathcal{C})$-cofibrant, so $q^{-1}(\tilde{b})$ is $m (\tilde{H} \cap \mathcal{C})$-cofibrant by Lemma \ref{hco60}, since it is equal to $p^{-1}(b)$ as an $\tilde{H}$-space. So we will assume from now on that $B$ is a $\mathcal{C}$-CW complex.
	
	Suppose that the $n$-skeleton, $B_{(n)}$, of $B$ is defined by attaching maps $\alpha_i: G/ H_i \times S^{n-1} \to B_{(n-1)}$, where $i$ ranges over some indexing set $I$ and for each $i$, $H_i \in \mathcal{C}$. Let $e_i$ denote the inclusion of the $i$th cell into $B_{(n)}$, $e_i: G / H_i \times D^n \to B_{(n)}$. Define $P_i^{S^{n-1}}$ to be the pullback of $\alpha_i$ along $p$, and $P_i^{D^n}$ to be the pullback of $e_i$ along $p$. Then, using Corollary \ref{hco62}, $P_i^{S^{n-1}} \simeq (G \times_{H_i} F_i) \times S^{n-1} $ and $P_i^{D^n} \simeq (G \times_{H_i} F_i) \times D^n$, where $F_i = p^{-1}(b_i)$ for some $b_i \in B$ with isotropy group $H_i$. Since $B$ is $\mathcal{C}$-cofibrant and $b_i \in B$, $H_i \in \mathcal{C}$. Moreover, by assumption, $F_i$ is $m (H_i \cap  \mathcal{C})$-cofibrant, so it follows that $P_i^{S^{n-1}}$ and $P_i^{D^n}$ are $m \mathcal{C}$-cofibrant by Lemma \ref{hco63} and the fact that $\mathcal{C}$ is closed under intersections. Using \cite[Lemma 6.2.3]{R23}, we have a pushout:
	
	\[ \begin{tikzcd}
		\sqcup_i P_i^{S^{n-1}} \arrow{d} \arrow{r} & p^{-1}(B_{(n-1)}) \arrow{d} \\ \sqcup_i P_i^{D^n} \arrow{r} & p^{-1}(B_{(n)}) \\
	\end{tikzcd} \]
	Since the pullback of an $hG$-cofibration along an $hG$-fibration is an $hG$-cofibration, \cite[Lemma 1.3.1]{MP12}, $P_i^{S^{n-1}} \to P_i^{D^n}$ is an $hG$-cofibration. Therefore, the LHS map is an $hG$-cofibration between $m \mathcal{C}$-cofibrant objects and, hence, an $m \mathcal{C}$-cofibration, by \cite[Proposition 3.11]{C06}. Therefore, the RHS is also an $m \mathcal{C}$-cofibration, and so, inductively, $p^{-1}(B) = E$ is $m \mathcal{C}$-cofibrant.		
\end{proof}

\section{Appendix of point set topology}

\begin{lemma} \label{hco9}
	Let $I$ be a filtered category and let $D: I \to \mathcal{U}$ be a functor which takes maps in $I$ to closed inclusions. Let $A := colim D$ and $f: X \to A$ be a map. Define $Y(i)$ to be the pullback of $D(i) \to A$ along $f$. Then the canonical map $colim Y \to X$ is a homeomorphism.
\end{lemma}

\begin{proof}
	We can express filtered colimits of closed inclusions, such as $A$, as quotients of the disjoint union of the colimiting subspaces, $\sqcup_{i \in I} D(i) \to A$, \cite[Lemma 3.3]{S09}. We have a pullback:
	
	\[	\begin{tikzcd}
		\sqcup_{i \in I} Y_i  \arrow{d} \arrow{r} & X \arrow{d}{f} \\
		\sqcup_{i \in I} D_i \arrow[swap]{r} & A 		
	\end{tikzcd} \]
	
	Since the pullback of a quotient map is a quotient map, \cite[Propostion 2.36]{S09}, we obtain the result.
\end{proof}

\begin{lemma}
	The following properties of a space $X$ are closed under retracts: being i) Hausdorff, ii) paracompact, iii) completely regular.
\end{lemma}

\begin{proof}
	
	Suppose that we have maps $i: A \to X, r: X \to A$ with $ri = 1$. If $X$ is Hausdorff, let $a$ and $b$ be distinct points in $A$. Then there exists disjoint open sets $U$ and $V$ in $X$ separating $i(a)$ and $i(b)$. So $i^{-1}(U)$ and $i^{-1}(V)$ are disjoint and separate $a$ and $b$. So $A$ is Hausdorff.
	
	 If $X$ is paracompact and $\{U_i\}_{i \in I}$ is an open cover of $A$, then $\{r^{-1}(U_i)\}$ is an open cover of $X$ which admits a locally finite open refinement $\{V_j\}_{j \in J}$. Then, for each $j$ $i^{-1}(V_j)$ is contained within $i^{-1}r^{-1}(U_i) = U_i$ for some $i$, so $\{i^{-1}(V_j)\}$ is an open refinement of $\{U_i\}$. Moreover, for every $a \in A$ there is an open neighbourhood, $W$, of $i(a)$ which intersects only finitely many of the $V_j$. Then $i^{-1}(W)$ is an open neighbourhood of $a$, intersecting only finitely many of the $i^{-1}(V_j)$. So $A$ is paracompact.
	
	Next suppose that $X$ is completely regular. Let $a \in A$ and $C$ be a closed subspace of $A$ not containing $a$. Then $r^{-1}(C)$ is a closed subspace of $X$ not containing $i(a)$. So there exists a continuous function $\lambda: X \to [0,1]$ such that $\lambda i (a) = 0$ and $\lambda(r^{-1}(C)) = 1 \implies \lambda i (C) = 1$. So $A$ is completely regular.
\end{proof}

\begin{lemma} \label{hco50}
	Let $C$ be a closed subset of $X$ and define an equivalence relation on $X \times I$ by $(x,t) \sim (y,s)$ iff $x = y$ and either $s = t$ or $x \in C$. Let $f: A \to X$ and $g: f^{-1}(X / C) \to I$ be continuous maps. Define a function $F: A \to (X \times I)/ \sim$ by:
	
	\begin{equation*}
		F(a) =
		\begin{cases}
			(f(a), g(a)) & \text{if} \ \ a \in f^{-1}(X/C) \\
			(f(a), 1) & \text{otherwise}
		\end{cases}
	\end{equation*}
	
Then $F$ is continuous.
\end{lemma}

\begin{proof}
	Let $V = X / C$ and $\pi: (X \times I) / \sim \ \to X$ be the projection. Since $\sim$ is a closed equivalence relation, $(X \times I) / \sim$ is CGWH with the usual quotient topology, \cite[pg. 40]{M99}. Since $I$ is compact Hausdorff, $\pi$ is a closed map. Let $E \subset (X \times I) / \sim$ be closed. Then $F^{-1}(E) \cap f^{-1}(V)$ is closed in $f^{-1}(V)$ by the continuity of $f$ and $g$. Moreover, $F^{-1}(E) \cap f^{-1}(V) \subset f^{-1}(\pi(E))$, which is closed since $\pi$ is a closed map and $f$ is continuous. Therefore, there exists a closed subset $B$ of $f^{-1}(\pi(E))$ such that $B \cap f^{-1}(V) = F^{-1}(E) \cap f^{-1}(V)$. Now $F^{-1}(E) \cap f^{-1}(C) = f^{-1}(C \cap \pi(E))$ is closed, since $C$ is closed. So:
	\[
	F^{-1}(E) = B \cup f^{-1}(C \cap \pi(E))
	\]
	is closed, so $F$ is continuous.
\end{proof}

\begin{lemma} \label{hco51}
	Let $i: A \to B$ be an $hG$-cofibration represented by an $G$-NDR-pair $(H, \lambda)$ and suppose that $j:B \to C$ is an $hG$-cofibration. Then there exists a $G$-NDR-pair, $(L, \tau)$, representing the composite $k:= ji: A \to C$ as an $hG$-cofibration such that $L|_{B \times I} = H$ and $\tau|_B = \lambda$.
\end{lemma}

\begin{proof} Let $(K,\mu)$ be an $G$-NDR-pair representing $j$ such that $K(c,t) = K(c,0)$ for all $t \leq 1 - \mu(c)$, which is possible by Lemma \ref{hco50}. Let $D$ be the closure of $\mu^{-1}([0,1))$ and define a map $E: C \times I \cup D \times [0,2] \to C$ by the concatenation of $K$ and $H$. Define:
	
\[ \tilde{\mu}(c) =	
\begin{cases}
		0 & \text{if} \ \ \mu(c) \leq \frac{1}{2} \\
		2(\mu(c) - \frac{1}{2}) & \text{if} \ \ \mu(c) \geq \frac{1}{2}
	\end{cases} \]

Now define $L(c,t) = E(c, (1-t)(1- \mu(c)) + t(2 - \tilde{\mu}(c)))$ and $\tau(c) = \min \{1, \max \{2\mu(c), \lambda(K(c,1))\} \}$, if $\mu(c) < 1$, and $\tau(c) = 1$ otherwise.
\end{proof}

\begin{lemma} \label{hco52}
	Given a pushout square:
	
		\[
	\begin{tikzcd}
		A \arrow{r}{f} \arrow[swap]{d}{i} & C \arrow{d}{j} \\
		B \arrow[swap]{r}{g} & D 
	\end{tikzcd}
	\]
	
	where $i$ is an $hG$-cofibration represented by the $G$-NDR-pair $(H, \lambda)$, then there exists a $G$-NDR-pair $(K, \mu)$ representing $j$ as an $hG$-cofibration such that $K \circ (g \times 1) = H$ and $\mu g = \lambda$.
\end{lemma}

\begin{proof}
	We can define $K$ using the relations $K \circ (g \times 1) = H$ and $K \circ (j \times 1) = \pi_C$ and the universal property of pushouts. Similarly for $\mu$.
\end{proof}

We also record the following lemma which follows from the fact that the $\Sigma_n$-equivariant NDR-pair given in the proof of \cite[Lemmma A.4]{M72} is also $G^n$-equivariant:

\begin{lemma} \label{hco56}
	If $i:A \to B$ is an $hG$-cofibration, then the $n$-fold pushout-product $i^{\hat{\times}n} : A \times B^{n-1} \cup ... \cup B^{n-1} \times A \to B^n$ is an $h(G \wr \Sigma_n)$-cofibration.
\end{lemma}

	\bibliography{References}
	\bibliographystyle{alpha}

\end{document}